\documentclass[10pt]{amsart}

\usepackage{amsmath}

\usepackage{bbm}
\usepackage{enumerate}
\usepackage{paralist}
\usepackage{amssymb}

\usepackage{stmaryrd}

\newtheorem{theorem}{Theorem}

\newtheorem{lemma}[theorem]{Lemma}
\newtheorem{corollary}[theorem]{Corollary}

\theoremstyle{definition}

\theoremstyle{remark}

\newcommand \NN{\mathbbm{N}}

\newcommand \ZZ{\mathbbm{Z}}

\newcommand {\gen}[1]{\langle #1 \rangle}
\newcommand {\clg}[1]{\overline{\left\langle #1 \right\rangle}}

\begin{document}

\title[Abelian pro-countable groups]{Abelian pro-countable groups and non-Borel orbit equivalence relations}

\author{Maciej Malicki}

\address{Department of Mathematics and Mathematical Economics, Warsaw School of Economics, al. Niepodleglosci 162, 02-554,Warsaw, Poland}
\email{mamalicki@gmail.com}

 
\begin{abstract}
We study topological groups that can be defined as Polish, pro-countable abelian groups, as non-archimedean abelian groups or as quasi-countable abelian groups, i.e., Polish subdirect products of countable, discrete groups, endowed with the product topology. We characterize tame groups in this class, i.e., groups all of whose continuous actions on a Polish space induce a Borel orbit equivalence relation, and relatively tame groups, i.e., groups all of whose diagonal actions $\alpha \times \beta$ induce a Borel orbit equivalence relation, provided that $\alpha$, $\beta$ are continuous actions inducing Borel orbit equivalence relations.
\end{abstract}

\maketitle
\section{Introduction}

Let G be a Polish group acting on a Polish space $X$ via a continuous action $\alpha:G \times X \rightarrow X$. Define
\[ x \, E_\alpha y \ \Leftrightarrow \ \exists g \in G \ \alpha(g,x) = y \]
for $x , y \in X$. Then $E_\alpha \subseteq X \times X$ is an equivalence relation called the \emph{orbit equivalence relation} induced by the action $\alpha$. It is natural to ask to what extent the topological and algebraic structure of $G$ determines the descriptive complexity of $E_\alpha$. Thus, a Polish group $G$ is called \emph{tame} if every orbit equivalence relation induced by a continuous action of $G$ is Borel. It is called \emph{relatively tame} if every diagonal action $\alpha \times \beta$ of $G$ induces a Borel orbit equivalence relation, provided that the actions $\alpha$ and $\beta$ of $G$ are continuous, and they induce Borel orbit equivalence relations.

It is well known that $E_\alpha$ is analytic for every Polish group $G$; if additionally $G$ is locally compact, $E_\alpha$ must be Borel, that is, $G$ is tame. As far as non-locally compact groups are concerned, S.Solecki \cite{Sol2} characterized tameness for Polish groups of the form $\prod_n G_n$, where each $G_n$ is abelian and discrete. Later, L.Ding and S.Gao \cite{DiGa} proved that in this class a group is tame if and only if it is relatively tame.

In this paper, we obtain analogous results for a larger class of abelian groups which can be characterized in the following three ways. First, they can be defined as Polish, \emph{pro-countable} abelian groups, that is, inverse limits of countable systems of discrete, countable abelian groups. Also, they can be viewed as \emph{non-archimedean} abelian  groups, that is, Polish abelian groups with a neighborhood basis at the identity consisting of open subgroups. Finally, they can be defined as \emph{quasi-countable} abelian groups, that is, closed, countable subdirect products of discrete, countable abelian groups, endowed with the product topology. Actually, we will be referring to them as quasi-countable groups because this definition is the most natural, and the most general in case we want to put some additional requirements on their structure. For example, in this note we will be also considering quasi-divisible groups, quasi-reduced groups, etc.

We prove (in Theorem \ref{th:Main:Red}) that an abelian quasi-reduced group $G$ is tame if it contains an open quasi-torsion group $H$ such that all Sylow subgroups of $H$ are locally compact. Otherwise, $G$ is not tame nor relatively tame. Because every abelian quasi-countable group $G$ can be written as $G=R \oplus \prod_n D_n$, where $R$ is quasi-reduced, and $D_n$, $n \in \NN$, are discrete and divisible (see Lemma \ref{le:Decom}), this gives rise to a complete characterization of tameness and relative tameness for abelian quasi-countable groups, provided in Theorem \ref{th:Main}.
 
\section{Terminology and basic facts}
All the groups considered in this paper are abelian. A topological space is \emph{Polish} if it is separable and completely metrizable. A topological group is Polish if its group topology is Polish. It is well known (see \cite[Theorem 3.5.2]{Gao} and \cite[Exercise 3.5.1]{Gao}) that every continuous action $\beta$ of a Polish subgroup $H $ of a Polish group $G$ on a Polish space $X$ gives rise to a continuous action $\alpha$ of $G$ such that $E_\beta$ is Borel iff $E_\alpha$ is Borel.

Let $\mathcal{D}$ be a class of discrete, countable groups. A Polish group $G$ is called \emph{quasi}-$\mathcal{D}$ if it is a subdirect product of groups from $\mathcal{D}$, that is,
\[ G \leq \prod_n G_n, \]
where $G_n \in \mathcal{D}$ for $n \in \NN$, $\prod_n G_n$ is endowed with the product topology, and all the projections of $G$ on $G_n$ are surjective. Thus, we can talk about quasi-countable groups, quasi-divisible groups, quasi-dsc groups (where dsc stands for `direct sum of finite cyclic groups'), quasi-reduced groups (recall that an abelian group is reduced if it does not contain a non-trivial divisible group) or quasi-$p$-reduced groups (that is,  groups with an adequate family consisting of reduced $p$-groups.) Similarly, a Polish group $G$ is called pro-$\mathcal{D}$ if it is the inverse limit of an inverse system of groups from $\mathcal{D}$. Clearly, every Polish pro-$\mathcal{D}$ group is quasi-$\mathcal{D}$.

In \cite{Ma} (see also \cite{GaXu}), where quasi-countable groups are more thoroughly discussed, the following lemma has been proved.

\begin{lemma}
\label{leCharPro}
Let $G$ be a Polish (not necessarily abelian) group. The following conditions are equivalent.

\begin{enumerate}
\item $G$ is quasi-countable,
\item $G$ is pro-countable,
\item $G$ has a neighborhood basis at the identity consisting of open, normal subgroups,
\item $G$ has a neighborhood basis at the identity consisting of open subgroups, and there exists a compatible two-sided invariant metric on $G$. 
\end{enumerate}
\end{lemma}

Let  $G \leq \prod_n G_n$. For $g \in G$, by $g(n)$ we mean the value of $g$ on its $n$th coordinate, $\pi_n[G]$ denotes the projection of $G$ on $G_n$, $G_{\gen{n}}$ denotes the group
\[ G_{\gen{n}}=\{g \in G: g(k)=1 \mbox{ for } k \leq n \},\]
and $G_{\gen{-1}}=G$. Clearly, in the product topology, the family $\{G_{\gen{n}}\}$ forms a neighborhood basis at the identity for $G$.

If $\pi_n[G]=G_n$ for each $n$, we say that the family $\{G_n\}$ is \emph{adequate for} $G$. 

An element $g$ of a quasi-countable group $G$ is called pro-$p$, if $\clg{g}$ is a pro-$p$ group or, equivalently, there is a prime $p_0$ such that the order of $g(n)$ is a power of $p_0$ for every $n$. If $G$ is abelian, and $p_0$ is a fixed prime, the $p_0$-\emph{Sylow} subgroup of $G$ is defined as the group of all pro-$p_0$ elements in $G$. This agrees with the standard terminology used in the theory of pro-finite groups (see \cite{RiZa}.) Observe that every Sylow subgroup of $G$ is closed.  

The cyclic group of order $n$ is denoted by $\mathbbm{Z}(n)$. For a fixed prime $p$, the \emph{Pr\"{ufer} group} $\mathbbm{Z}(p^\infty)$ is the unique $p$-group in which the number of $p$th roots of every element is exactly $p$. 

Let $G$ be a $p$-group. The \emph{rank of} $G$, denoted by r$(G)$, is the cardinality of a maximal independent system in $G$ containing only elements of prime power orders (see \cite[p. 83]{Fu}). The \emph{final rank of} $G$, denoted by fin r$(G)$, is defined as
\[ \mbox{fin r}(G)=\inf \{ \mbox{r}(p^nG): n \in \NN \}. \]

In the sequel we will use the following results relating tameness, and relative tameness to algebraic properties of groups.

\begin{theorem}[Solecki]
\label{th:Sol2}
Let $G_n$, $n \in \NN$, be countable, discrete abelian groups. Then $\prod_n G_n$ is tame iff $G_n$ is torsion for all but finitely many $n$, and for all primes $p$ for all but finitely many $n$ the $p$-Sylow subgroup of $G_n$ is of the form $F \oplus \mathbbm{Z}(p^\infty)^k$, where $k \in \NN$, and $F$ is a finite abelian p-group.
\end{theorem}

This characterization implies that if, say, each $G_n$ is a finite direct sum of Pr\"{u}fer groups $\mathbbm{Z}(p^\infty)$, then $\prod_n G_n$ is tame, and thus all Polish subgroups of $\prod_n G_n$ are also tame.  On the other hand, if, for some fixed $p_0$, each $G_n$ is the infinite direct sum of $\mathbbm{Z}(p_0^\infty)$, then $\prod_n G_n$ is not tame, but this does not necessarily imply that all Polish, non-locally compact subgroups of $\prod_n G_n$ are not tame. 

\begin{theorem}[Ding, Gao] \label{th:DiGa}
Let $G_n$, $n \in \NN$, be countable, discrete abelian groups. Then $\prod_n G_n$ is relatively tame if and only if it is tame.
\end{theorem}

\section{Main results}

\begin{lemma} \label{le:Rank}
Let $A$ be a countable, reduced $p$-group. Then either $A$ is a bounded dsc group or it has infinite final rank.
\end{lemma}

\begin{proof}
Suppose that the final rank of $A$ is finite. Then there exists $n \in \NN$ such that the rank of the subgroup $nA \leq A$ is finite. Since every countably infinite, reduced group can be decomposed into a direct sum of infinitely many non-trivial groups (see \cite[Proposition 77.5]{Fu}), $nA$ must be finite as well. But this means that there exists $n'>n$ such that $n'A$ is trivial, that is, $A$ is bounded. By \cite[Theorem 17.2]{Fu}, $A$ is a direct sum of cyclic groups. 
\end{proof}

Let us state two results from \cite{Ma} that are relevant in the present context.

\begin{lemma}[Lemma 14 in \cite{Ma}]
\label{leKul}
Let $G$ be a non-locally compact, quasi-torsion group. Then either there exists an open quasi-dsc subgroup $H \leq G$ or there exists a closed subgroup $L \leq G$ such that $G/L$ is quasi-divisible, and non-locally compact.
\end{lemma}

\begin{theorem}[Lemma 15 and Theorem 17 in \cite{Ma}]
\label{thProC}
Let $G$ be a non-locally compact, quasi-countable group. Then there exists a closed $L \leq G$, and infinite, discrete  groups $K_n$, $n \in \NN$, such that 
\[ \prod_n K_n \leq G / L. \]
Moreover, all $K_n$ are bounded (dsc) if $G$ is a quasi-bounded (quasi-dsc) group.
\end{theorem}

It turns out that in case $G$ is quasi-p-reduced, slightly more can be proved.

\begin{lemma} \label{le:Red}
Let $G$ be a non-locally compact, quasi-$p$-reduced group. Then either there exists an open quasi-bounded-dsc subgroup $H \leq G$ or there exists a closed subgroup $L \leq G$, and discrete, divisible groups $K_n$, $n \in \NN$, with infinite rank, and such that 
\[ \prod_n K_n \leq G/L. \]
\end{lemma}

\begin{proof}
Let $\{G_n\}$ be an adequate family for $G$ consisting of reduced $p$-groups. By Lemma \ref{le:Rank}, either there exists $m$ such that $\pi_n[G_{\gen{m}}]$ is a bounded dsc group for every $n>m$ or for every $m$ there exists $n>m$ such that $\pi_n[G_{\gen{m}}]$ has infinite final rank. If the former holds, $G_{\gen{m}}$ is quasi-bounded-dsc. Otherwise, by \cite[Theorem 35.6]{Fu}, for every $m$ there exists $n>m$ and $L'_n \leq G_n$ such that $K'_n=\pi_n[G_{\gen{m}}]/L'_n$ is a divisible group with infinite rank. Since divisible groups are absolute direct summands, for every $m$ we can find $f(m) \in \NN$, and $L''_{f(m)} \leq G_{f(m)}$ such that 
\[ G_{f(m)}/L'_{f(m)}=K'_{f(m)} \oplus (L''_{f(m)}/L'_{f(m)}). \]
Put $L_n=\langle L'_n, L''_n \rangle$ if $n$ is of the form $n=f^m(0)$ for some $m$, and $L_n=G_n$ otherwise. It is straightforward to verify that $K_n=K'_{f^n(0)}$, $L=(\prod_n L_n) \cap G$ are  as required.
\end{proof}
 
\begin{corollary} \label{co:Red}
Let $G$ be a non-locally compact, quasi-$p$-reduced group. Then there exists a closed subgroup $L \leq G$, and infinite, discrete groups $K_n$, $n \in \NN$, all of whom are either bounded dsc groups or divisible groups with infinite rank, and that are such that 
\[ \prod_n K_n \leq G/L. \]
\end{corollary}

\begin{proof}
If $G$ contains an open quasi-bounded-dsc subgroup, we use Theorem \ref{thProC}. Otherwise, we apply Lemma \ref{le:Red}.
\end{proof}

\begin{theorem} \label{th:Main:Red}
Let $G$ be a non-locally compact, quasi-reduced group. Then exactly one of the following holds.
\begin{enumerate}
\item in every neighborhood of the identity in $G$ there exists an element generating an infinite, discrete group; in this case $G$ is not tame nor relatively tame,
\item there exists an open quasi-torsion $H \leq G$, and a non-locally compact Sylow subgroup of $H$; in this case, $G$ is not tame nor relatively tame,
\item there exists an open quasi-torsion $H \leq G$ such that every Sylow subgroup of $H$ is locally compact; in this case, $G$ is tame.
\end{enumerate}
\end{theorem}

\begin{proof}
Suppose that the assumption of Point (1) holds. An inspection of the proof of  \cite[Corollary 12]{Ma} shows that $\ZZ^\NN \leq G$. Thus, Point (1) follows from Theorems \ref{th:Sol2} and \ref{th:DiGa}, and the fact that every continuous action $\beta$ of $\ZZ^\NN$ extends to a continuous action $\alpha$ of $G$ such that $E_\alpha$ is not Borel if $E_\beta$ is not Borel.

Suppose now that the assumption of Point (1) does not hold, that is, there exists an open quasi-torsion subgroup in $G$. Without loss of generality, we can assume that $G$ is quasi-torsion. Let $\{G_n\}$ be an adequate family for $G$, and let $\{S_p\}$ be the family of all Sylow subgroups of $G$. If $S_p$ is non-locally compact for some prime $p$, we apply Corollary \ref{co:Red}, and Theorem \ref{th:Sol2}.

Otherwise, every $p$-Sylow subgroup $S_p$ of $G$ is locally compact. Let $\{p_k\}$ be an enumeration of all prime numbers. We can easily construct a strictly increasing sequence $\{i_k\}$ such that each $K_k=G_{\gen{i_k}} \cap S_{p_k}$ is compact. By \cite[Lemma 8]{Ma}, $K=\prod_k K_k$ is a compact subgroup of $G$.

Let $\{G'_n\}$ be an adequate family for $G/K$. As every $p$-Sylow subgroup $S'_p$ of $G/K$ is discrete, for every $p$ we can find $n_p$ such that, for all $g,h \in S'_p$, $g(n)=h(n)$ for $n<n_p$ implies that $g=h$. In particular, without loss of generality we can assume that in $G/K$ we have that $\pi_n[S'_p]$ is trivial for $n \geq n_p$ (see \cite[Lemma 5]{Ma} for details.) Thus, by Theorem \ref{th:Sol2}, both $\prod_n G'_n$ and $G/K$ are tame.

Suppose now that $\alpha$ is a continuous action of $G$ on a Polish space $X$. Let $O$ be the set of all the orbits $[x]_K$, $x \in X$, of the restriction of $\alpha$ to $K$. It is not hard to see that $O$ is a Borel subset of the hyperspace $\mathcal{K}(X)$ of all compact subsets of $X$ with the Vietoris topology. Moreover, $\alpha$ naturally gives rise to a continuous action $\beta$ of $G/K$ on $O$. To see this, fix $g_0,g_1 \in G$ such that $g_0K=g_1K$, that is, $g_0=g_1 h$ for some $h \in K$. Then
\[ g_0.x=g_1.(h.x)=h.(g_1.x), \]
for $x \in X$, that is, 
\[ g_0.[x]_K=g_1.[x]_K=[g_1.x]_K \]
for $x \in X$. Also, $E_\beta$ is Borel because $G/K$ is tame, and $\beta$ can be regarded as a continuous action on a Polish space (see \cite[Theorem 5.2.1]{Be}.) Checking that
\[ x E_\alpha y \ \Leftrightarrow \ [x]_K E_\beta [y]_K \]
for $x, y \in X$ is straightforward. As the mapping $x \mapsto [x]_K$ is continuous, and $\alpha$ was arbitrary, we get that $G$ is tame.
\end{proof}

Now we have the following decomposition result for (abelian) quasi-countable groups.

\begin{lemma}
\label{le:Decom}
Let $G$ be a quasi-countable group. Then $G=R \oplus \prod_n D_n$, where $R$ is quasi-reduced, and each $D_n$ is divisible.
\end{lemma}

\begin{proof}
Let $\{G_n\}$ be an adequate family for $G$. As divisible groups are absolute direct summands (see \cite[Theorem 21.2]{Fu}), we can write $G_n=R_n \oplus D'_n$, $n \in \NN$, where $R_n$ is reduced, $D'_n$ is divisible, and, moreover, the maximal divisible subgroup of $\pi_{n}[G_{\gen{m}}]$ is a subgroup of $D'_{n}$ for each $m$ and $n>m$. In other words, $G \leq \prod_n R_n \oplus \prod_n D'_n$, so, for
\[ R=G \cap \prod_n R_n, \, D'=G \cap \prod_n D'_n, \]
we get that $G=R \oplus D'$, where $R$ is quasi-reduced, $D'$ is quasi-divisible, and $\{D'_n\}$ is an adequate family for $D'$ such that $\pi_{n}[D'_{\gen{m}}]$ is divisible for every $m$ and $n>m$.

We construct groups $D_n$ such that $D'$ is isomorphic to $\prod_n D_n$ by induction. Put $D_0=D'_0$, and suppose that the projection of $D'$ on $D'_0 \oplus \ldots \oplus D'_n$ is surjective for some $n>0$. As divisible groups are absolute summands, we can write
\[ D'_{n+1}=E \oplus F, \]
where $F=\pi_{n+1}[D'_{\gen{n}}]$. Fix $g,h \in D' \setminus D'_{\gen{n}}$. Clearly, if $g(i)=h(i)$ for $i \leq n$, that is, $g-h \in D'_{\gen{n}}$, the restrictions of $g(n+1)$, $h(n+1)$ to $E$ are equal. Thus, we can assume that $E=\{0\}$, that is, that the projection of $D'$ on $D'_0 \oplus \ldots \oplus D'_{n+1}$ is also surjective. Put $D_{n+1}=F$.
\end{proof}

Thus, Theorem \ref{th:Main:Red} characterizes tameness and relative tameness for $R$, while for $\prod_n D_n$ it is done in Theorem \ref{th:Sol2}. We are in a position to prove our main result that completely characterizes tameness and relative tameness for all (abelian) quasi-countable groups.

\begin{theorem}
\label{th:Main}
Suppose that $G$ is a quasi-countable group. Let $R$ be a quasi-reduced group and let $D_n$, $n \in \NN$, be discrete, divisible groups such that $G=R \oplus \prod_n D_n$. If both $R$ and $\prod_n D_n$ are tame, then $G$ is tame. Otherwise, $G$ is not tame nor relatively tame.
\end{theorem}

\begin{proof}
Suppose that $R$ and $\prod_n D_n$ are tame. In particular, $\prod_n D_n$ satisfies the assumptions of Theorem \ref{th:Sol2}. By Theorem \ref{th:Main:Red}, we can assume without loss of generality that $R$ is quasi-torsion, and all Sylow subgroups of $R$ are locally compact. Arguing exactly as in the proof of Theorem \ref{th:Main:Red}, we find a compact $K \leq R$, and an adequate family $\{R'_n\}$ for $R/K$ such that $\pi_n[S_p]=\{0\}$ for every $p$, and almost all $n$. But this implies that $\prod_n R'_n \oplus \prod_n D_n$ satisfies the assumptions of Theorem \ref{th:Sol2}, and so $G/K$ is tame. Again, we get that $G$ is tame as well.
\end{proof}

\end{document}